\documentclass{amsart}
\usepackage[english]{babel}
\usepackage{amssymb,enumerate,amsmath,txfonts}
 \usepackage{graphicx}
\usepackage[colorlinks=true,linkcolor=blue,citecolor=blue,urlcolor=blue]{hyperref}
\numberwithin{equation}{section}

\newtheorem{theo}{Theorem}[section]

\newtheorem{lem}[theo]{Lemma}

\theoremstyle{remark}

\newtheorem{rem}[theo]{Remark}

\renewcommand{\(}{\left(}
\renewcommand{\)}{\right)}

\renewcommand{\d}{\delta}
\newcommand{\e}{\epsilon}

\renewcommand{\l}{\lambda}

\begin{document}
\title[1st Eigenvalue Pinching]{1st Eigenvalue Pinching for Convex Hypersurfaces in a Riemannian Manifold}
\author{Yingxiang Hu, Shicheng Xu}
\address{Yau Mathematical Scieneces Center, Tsinghua University, Beijing, China}
\email{huyingxiang@mail.tsinghua.edu.cn}
\address{Mathematics Department, Capital Normal University, Beijing, China}
\email{shichengxu@gmail.com}
\date{\today}
\begin{abstract}
	Let $M^n$ be a closed convex hypersurface lying in a convex ball $B(p,R)$ of the ambient $(n+1)$-manifold $N^{n+1}$. We prove that, by pinching Heintze-Reilly's inequality via sectional curvature upper bound of $B(p,R)$, 1st eigenvalue and mean curvature of $M$, not only $M$ is Hausdorff close and almost isometric to a geodesic sphere $S(p_0,R_0)$ in $N$, but also its enclosed domain is $C^{1,\alpha}$-close to a geodesic ball of constant curvature.
\end{abstract}
{\maketitle}

\section{Introduction}

As a natural extension of Reilly's result \cite{Reilly1977} in $\mathbb R^{n+1}$, Heintze \cite{Heintze1988} proved that a closed submanfold $M^n$ (including high codimension) lying in a convex ball $B(p,R)$ with $R\le\frac{\pi}{4\sqrt{\d}}$ of $N^{n+m}$, whose sectional curvature $K_N \leq \d$, satisfies
\begin{align}\label{1steigen-ineq}
\l_1(M)&\leq \begin{cases}
n\d+\frac{n}{|M|}\int_{M}|H|^2, &\text{for $\d\geq 0$;}\\
n\d+n\|H\|_\infty^2, &\text{for $\d<0$,}
\end{cases}
\end{align}
where $\l_1(M)$ is the first nonzero eigenvalue of the Laplace-Beltrami operator on $M$, $\|H\|_\infty=\max_M |H|$ and $|M|$ its area. Here $\frac{\pi}{\sqrt{\d}}$ is viewed to be $\infty$ if $\d\le 0$. Equality holds in \eqref{1steigen-ineq} if and only if $M$ is minimally immersed in some geodesic sphere of $N$.

Let $s_\delta$ be the usual $\delta$-sine function defined on $[0,\frac{\pi}{\sqrt{\d}})$ by
\begin{equation*}
\begin{split}
s_\d(r):=\left\{ \begin{aligned}
&\frac{1}{\sqrt{\d}}\sin (\sqrt{\d} r), \quad &\text{if}&~\d>0;\\
&r,                                     \quad &\text{if}&~\d=0;\\
&\frac{1}{\sqrt{-\d}}\sinh(\sqrt{-\d}r),\quad &\text{if}&~\d<0.
\end{aligned}
\right.
\end{split}
\end{equation*} and $s_\d^{-1}$ be its inverse function. Let $\omega_n$ be the volume of $n$-sphere of radius $1$ in $\mathbb R^{n+1}$. As observed in \cite{Hu-Xu2019},
if equality holds in \eqref{1steigen-ineq} for an immersed hypersurface $M^n$ ($n\ge 2$) in $N^{n+1}$, then $M$ encloses a geodesic ball of constant curvature $\d$. We prove in \cite{Hu-Xu2019} that if \eqref{1steigen-ineq} almost holds for $\|H\|_\infty^2$, i.e.,
\begin{align}\label{pinching-condition-1}
n(\d+\|H\|_\infty^2) \leq \l_1(M)(1+\e),
\end{align}
then $M$ is Hausdorff close to a geodesic sphere $S(p_0,R_0)$, and $B(p_0,R_0)$ is almost to be of constant $\d$, provided with $R\le\frac{\pi}{8\sqrt{\d}}$, $|M|\le \omega_ns_\d^n(\frac{\pi}{4\sqrt{\d}})$, and a rescaling invariant control on the mean curvature and area of $M$: 
\begin{equation}\label{rescaling-inv-non-local-col}
|M|^\frac{1}{n}\|H\|_\infty\leq A.
\end{equation} 

According to \cite{Hu-Xu2019}, \eqref{rescaling-inv-non-local-col} can be viewed as a relative non-collapsing condition for hypersurfaces, i.e., 
$$\frac{|B(x_0,r)\cap M|}{|M|}\ge C(n,\d)r^n, \quad \text{for any $x_0\in M$ and $0<r\le \min\left\{\operatorname{diam}_N(M), \frac{\pi}{2\sqrt{\d}}\right\} $},$$
where $\operatorname{diam}_N(M)$ is the extrinsic diameter of $M$ as a subset in $N$.

Therefore, typical examples that do not satisfies \eqref{rescaling-inv-non-local-col} contain the boundary $\partial U_r$ of $r$-neighborhoods of a high co-dimensional submanifold $X$ (not a point) with $r\ll\operatorname{diam}_NX$.

In this paper, we prove that if $\partial U$ is convex, then $\partial U$ still cannot satisfy pinching condition (\ref{pinching-condition-1}). We use $\varkappa(\epsilon\,|\,R,\delta,\cdots)$ to denote a positive function on variables $\epsilon, R,\delta,\cdots$ that converges to $0$ as $\epsilon\to 0$ with other quantities $R,\delta,\cdots$ fixed.

\begin{theo}\label{main-theo-C}
	Let $n$ be an integer $\ge 2$ and let $N^{n+1}$ be a complete Riemannian manifold with $\mu \leq K_N\leq \d$. Let $M^n$ be an embedded convex hypersurface in a convex geodesic ball $B(p,R)$ of $N$. If $\d>0$ we further assume
	\begin{equation}\label{location-size}
	R\le \frac{\pi}{8\sqrt{\d}},\quad |M|\le \omega_n s_\d^n\(\frac{\pi}{4\sqrt{\d}}\).
	\end{equation}
	Then (\ref{pinching-condition-1}) with $\e<\e_0=\e_0(\mu,\d,n, R)$ implies that
	\begin{enumerate}\numberwithin{enumi}{theo}
		\item $M$ is $C\e^\frac{1}{2(2n+1)} s_\d(R_0)$-Hausdorff-close and $C\e^{\frac{1}{2(2n+1)}}$-almost isometric to a geodesic sphere $S(p_0,R_0)$, where $R_0=s_\d^{-1}(\frac{1}{\sqrt{\d+\|H\|_{\infty}^2}})$, and $C=C(n,\d,\mu, R)$ is a positive constant;
		\item $M$ is $\varkappa(\epsilon|R,\d,\mu,n,\alpha)$-$C^{1,\alpha}$ close to a round sphere of constant curvature $1/(s_\d(R_0))^2$ with $0\le\alpha<1$;
		\item The enclosed domain $\Omega$ by $M$ is $\varkappa(\epsilon\,|\, R,\d,\mu,n,\alpha)$ $C^{1,\alpha}$-close to a ball of constant curvature $\delta$.
	\end{enumerate}
Here $\epsilon_0$, $C$, $\varkappa$ above do not depend on $R$ when $\d\ge 0$.
\end{theo}

By Theorem \ref{main-theo-C}, one cannot obtain a collapsed metric $g$ via perturbing the interior of a geodesic ball $\Omega$ of constant curvature $\d$ to a new metric $g_\d$ whose sectional curvature's upper bound is close to $\d$.
 
We point it out that the extrinsically relative non-collapsing condition \eqref{rescaling-inv-non-local-col} are required in all earlier known results via pinching \eqref{rescaling-inv-non-local-col}. Here for Theorem \ref{main-theo-C} we also prove that convexity of $M$ plus the pinching condition (\ref{pinching-condition-1}) implies  (\ref{rescaling-inv-non-local-col}) holds for some universal $A$.

At present quantitative rigidity results with respect to the upper curvature bound $\d$ are rarely known. There are relative more rigidity results for convex domains to be of constant curvature via information along their boundary and interior's lower (Ricci) curvature bound (e.g., \cite[Theorems 1, 2]{Schroeder-Strake1989}, \cite[Theorem 1.2]{Miao-Wang2016}, etc.). Compared to Theorem \ref{main-theo-C},  few of their quantitative versions are known.

We refer to \cite{Hu-Xu2019}, \cite{Grosjean-Roth2012},  \cite{Aubry-Grosjean}, \cite{Colbois-Grosjean2007}, and \cite{Hu-Xu2017} for related results in manifolds and space forms.

Let us give the main idea in proving Theorem \ref{main-theo-C}.
If there is a universal constant $A=A(n,\d,\mu)$ such that \eqref{rescaling-inv-non-local-col} holds for a convex hypersurface $M$ in Theorem \ref{main-theo-C}, then it is a direct corollary of the following and Cheeger-Gromov's convergence Theorem ( \cite{Cheegerphd,Cheeger1970}, \cite{GLP1981}, \cite{Kasue1989}, \cite{Green-Wu1988}, \cite{Peters1987}) that conclusions in Theorem \ref{main-theo-C} hold.

\begin{theo}[\cite{Hu-Xu2019}]\label{main-theo-B}
	Let $n$ be an integer $\ge 2$, and let $N^{n+1}$ be a complete Riemannian manifold with $\mu \leq K_N\leq \d$. Let $M^n$ be an immersed hypersurface in a convex geodesic ball $B(p,R)$ of $N$, where $R$ and $|M|$ satisfy \eqref{location-size} when $\d > 0$. If for $q\ge \bar q>n$ and $A>0$, we have
	\begin{equation}\label{rescaling-inv-bound}
	|M|^\frac{1}{n}\|H\|_\infty\leq A, \qquad \text{and} \qquad
	|M|^\frac{1}{n}\|B\|_q\leq A,
	\end{equation}
	where $\|B\|_q=\left(\frac{1}{|M|}\int |B|^q\right)^{1/q}$ is the normalized $L^q$ norm of the 2nd fundamental form $B$ of $M$.
	Then (\ref{pinching-condition-1}) with $\e<\e_1(A,\bar q, R,\d,\mu,n)$ implies that
	\begin{enumerate}\numberwithin{enumi}{theo}
		\item $M$ is $C_1\e^\frac{1}{2(2n+1)} s_\d(R_0)$-Hausdorff-close to a geodesic sphere $S(p_0,R_0)$ in $N$, where $R_0=s_\d^{-1}(\frac{1}{\sqrt{\d+\|H\|_{\infty}^2}})$;
		\item $B(p_0,R_0)$ is $\varkappa(\epsilon)$ $C^{1,\alpha}$-close to a ball of constant curvature $\delta$ for any $0\le \alpha<1$;
		\item $M$ is embedded and $C_2\e^{\min\left\{ \frac{1}{2(2n+1)},\frac{q-n}{2(q-n+qn)}\right\}}$-almost isometric to $S(p_0,R_0)$.
	\end{enumerate}
Here constants $C_1, C_2$ and $\epsilon_1$ depend on $n$, $\d$, $\mu$, $\bar q$, $R$, $A$, and the function $\varkappa(\epsilon)$ also depends on $n$, $\d$, $\mu$,  $\bar q$, $R$, $A$, $\alpha$, where the dependence on $R$ can be dropped when $\d\ge 0$.
\end{theo}

In order to obtain $A(n,\mu,\d)$ above, we argue by contradiction. Suppose there is a sequence of convex domains $\Omega_i$ such that $|\partial \Omega_i|^{\frac{1}{n}}\|H_{\partial \Omega_i}\|_\infty \to \infty$. We will prove that $\lambda_1(\partial \Omega_i)$ admits a universal upper bound after a rescaling such that the area $|\partial \Omega_i|=1$. On the other hand, by pinching condition \eqref{pinching-condition-1}, $\|H_{\partial \Omega_i}\|_{\infty}$ has a uniform upper bound, a contradiction. 

We will use gradient flows of semi-concave functions on Alexandrov spaces \cite{Burago-Gromov-Perelman1992} to derive a universal ratio upper bound on the intrinsic diameter of $\partial \Omega_i$, and then apply Cheeger-Colding's eigenvalue convergence theorem \cite{Cheeger-Colding2000} on $\partial \Omega_i$ to obtain an upper bound of $\lambda_1$. 

\section{Proof of Theorem \ref{main-theo-C}}
\label{sec:6}

In the following we will apply some results on Alexandrov spaces with curvature bounded below. The references of this section are  \cite{Burago-Gromov-Perelman1992}, \cite{Petrunin2007}.

Roughly speaking, a finite-dimensional {\em Alexandrov space $X$ of curv $\ge \mu$} is a locally compact length metric space on which Toponogov comparison theorem for triangles holds just as that on manifolds of sectional curvature $\ge \mu$. 

A locally Lipschitz function $f:U\subset X\to \mathbb R$ is called {\em $\lambda$-concave}, if $U$ is an open domain and there is a real number $\lambda$ such that for any unit-speed minimal geodesic $\gamma$ in $U$, $f\circ\gamma(t)-\frac{\lambda}{2}t^2$ is concave. 

The {\em gradient flow} $\Phi_t$ of a $\lambda$-concave function is well-defined, and is $e^{\lambda t}$-Lipschitz (see \cite[Lemma 2.1.4(i)]{Petrunin2007}). {\em Extremal subsets} are those subsets which are invariant under gradient flows of all $\lambda$-concave functions \cite{Perelman-Petrunin1993,Petrunin2007}. The boundary of an Alexandrov space of curv $\ge \mu$ is an extremal subset. 

A {\em quasigeodesic} in $X$ is a unit-speed curve $\gamma$ such that for any point $p$ in $X$, the function $f(t)=\int_0^{d(p,\gamma(t))}s_\mu(s)ds$ 
satisfies $f''\le 1-\mu f$ in the barrier sense.
By its definition, limit of quasigeodesics is also a quasigeodesic, see \cite{Perelman-Petrunin1994}. 

By the generalized Lieberman lemma \cite[Theorem 1.1]{Petrunin1997}, any shortest geodesic in an extremal subset is a quasigeodesic in the ambient space. 

We say that a sequence $(X_i,d_i)$ of metric spaces {\em $GH$-converges} to $(X,d)$ in Gromov-Hausdorff topology, if there are $\epsilon_i$-isometries $\psi_i:X_i\to X$ with $\epsilon_i\to 0$, i.e., for any $x,y\in X_i$, $|d_i(x,y)-d(\psi_i(x),\psi_i(y))|\le\epsilon_i$, and $\epsilon_i$-neighborhood of $\psi_i(X_i)$ covers $X$.

The following is the main technical lemma in this paper to derive a bound of intrinsic diameter of $M$, which is required in applying a compactness argument under Gromov-Hausdorff topology. 

\begin{lem}\label{lem-intrinsic-diam}
	Let $\Omega$ be a convex domain of diameter $\le D$ with a smooth boundary $M=\partial \Omega$ in $N^{n+1}$ whose sectional curvature $K_N\ge \mu$. Then the intrinsic diameter of $M$, $$\operatorname{diam}(M)\le C(n,\mu,D)\operatorname{diam}(\Omega),$$ where $C(n,\mu,D)$ is a positive constant depending on $n$, $\mu$ and $D$.
\end{lem}
\begin{proof}
	Let us argue by contradiction. Assume the contrary, then there is a sequence $\Omega_i$ of convex domains of diameter $\le d$, where the ratio $\operatorname{diam}(M_i)/\operatorname{diam}(\Omega_i)\to \infty$.
	
	Note that the closure $\overline\Omega_i$ of $\Omega_i$ is an Alexandrov space with curv $\ge \mu$ of Hausdorff dimension $n+1$, and its boundary $M_i=\partial \Omega_i$ is an extremal subset of $\Omega_i$. By passing to a subsequence, we may assume that $\overline\Omega_i$ $GH$-converges to an Alexandrov space $X$ with curv $\ge \mu$.
	
	Case 1. The Hausdorff dimension of $X$ equals $n+1$. That is, $\Omega_i$ is a non-collapsing sequence. Then Lemma \ref{lem-intrinsic-diam} is a direct corollary of \cite[Theorem 1.2]{Petrunin1997}, whose special case is that the boundary $\partial \Omega_i=M_i$ with the intrinsic metric $GH$-converges to $\partial X$, the boundary of $X$.
	
	Indeed, by the $GH$-convergence, $d\ge \lim_{i\to \infty}\operatorname{diam}(\Omega_i)=\operatorname{diam}(X)>0$. And thus the intrinsic diameter $\operatorname{diam}(\partial X)=\lim_{i\to \infty}\operatorname{diam}(M_i)=\infty.$
	On the other hand, by Gauss equation, $\partial \Omega_i$ is an Alexandrov space with curv $\ge \mu$. Hence $\partial X$ is also an Alexandrov space with the same lower curvature bound. By the compactness of $\partial X$, $\operatorname{diam}(\partial X)<\infty$, a contradiction.
	
	Case 2. The Hausdorff dimension of $X$ $<n+1$. By passing to a rescaled converging subsequence, we assume that $\operatorname{diam}(\Omega_i)\to 1$.
	
	Let $p_i,q_i$ be two points in $M_i$ such that the intrinsic distance $d_{M_i}(p_i,q_i)=\operatorname{diam}(M_i)$, and let $\gamma_i=[p_iq_i]$ be a minimal geodesic in $\Omega_i$
	connecting $p_i=\gamma_i(0)$ and $q_i$ and realizing their distance $d(p_i,q_i)$ in $\Omega_i$. By Ascoli theorem and passing to a subsequence, we may assume $\gamma_i$ converges to a minimal geodesic $\gamma$ in $X$, which connects $p$ and $q$. Then $p_i\to p$, $q_i\to q$, and the length of geodesics, $L(\gamma_i)\to L(\gamma)=d(p,q)$.
	
	(i) $0<d(p,q)\le 1$. Let us consider the distance function $\operatorname{dist}_{p_i}$ to point $p_i$. Then it converges to $\operatorname{dist}_p$. Let $\alpha_i$ be a gradient curve of $\operatorname{dist}_{p_i}$ such that $\alpha_i(0)\in M_i$ and $\alpha_i(0)$ converges to $\gamma(\epsilon)\neq p$ as $\overline{\Omega}_i$ $GH$-converges to $X$. Since $\Omega_i$ is collapsing, such $\alpha_i$ always exists for any $\epsilon>0$ and $\alpha_i$ lies in $M_i$.
	
	Since $\gamma$ is the gradient curve of $\operatorname{dist}_p$, by \cite[Lemma 2.1.5]{Petrunin2007}, the gradient curve $\alpha_i$ of $\operatorname{dist}_{p_i}$ and $L(\alpha_i)$ converges to $\gamma|_{[\epsilon,L(\gamma)]}$ and $L(\gamma)$ respectively. Therefore, the gradient curve $\alpha_i$ has length $\le L(\gamma)+\epsilon$ for all $i$ large.
	
	Let $\hat p_i$ and $\hat q_i$ be the endpoints of $\alpha_i$, then for $\epsilon\to 0$ we have $\hat p_i=\alpha_i(0)\to p$ and $\hat q_i\to q$.
	Now by the convexity of $\Omega_i$, we are able to use the gradient flow $\Phi_\epsilon$ of $\operatorname{dist}_{p_i}$ (resp. $\operatorname{dist}_{q_i}$) by a definite time $\le 2\epsilon$ to push the minimal geodesic $[\hat q_iq_i]$ (resp. $[\hat p_ip_i]$) to a curve  $\beta_{q_i}$ (resp. $\beta_{p_i}$) in $M_i$, whose length $\le e^{C(\mu)\cdot\frac{\epsilon}{d(p,q)} }\epsilon$. Since the joint curve by $\alpha_i$, $\Phi_t(p_i)$,  $\Phi_t(\hat p_i)$, $\Phi_t(q_i)$, $\Phi_t(\hat q_i)$, $\beta_{p_i}$ and $\beta_{q_i}$ gives rise to a curve of uniformly bounded length connecting $p_i$ and $q_i$ in $M_i$, a contradiction to $d_{M_i}(p_i,q_i)\to \infty$.

	(ii) $d(p,q)=0$, i.e., $p=q$. Since $\operatorname{diam}(X)=1$, there is a point $z\in X$ such that $d(z,p)=\frac{1}{2}$. Let $z_i\in \Omega_i$ be a sequence of points converging to $z$. Then for all $i$ large, the gradient flow $\Phi_2$ of $\operatorname{dist}_{z_i}$ pushes the minimal geodesic $[p_iq_i]$ in $\Omega_i$ to a curve $\alpha$ in $M_i$ with length $e^{C(\mu)}d(p_i,q_i)$. Then the gradient curves $\Phi_t(p_i)$ ($t\in [0,1]$), $\alpha$ and $\Phi_t(q_i)$ ($t\in [0,1]$) form a curve connecting $p_i$ and $q_i$ in $M_i$, whose length is bounded by $2+e^{C(\mu)}d(p_i,q_i)\to 2$. This contradicts to that $d_{M_i}(p_i,q_i)\to \infty$.
\end{proof}

\begin{rem}
	Note that for closed convex hypersurfaces in $\mathbb R^n$, $\operatorname{diam}(M)\le \frac{\pi}{2}\operatorname{diam}(\Omega)$, see \cite{Makuha1966,Makuha1967}, cf. \cite[Page 43]{Burago-Zalgaller1988}. It is interesting if one can give a sharp estimate on the constant in Lemma \ref{lem-intrinsic-diam}.
\end{rem}

\vspace{2mm}
\begin{proof}[Proof of Theorem \ref{main-theo-C}]
	~
	
	Since $M$ is the boundary of a convex domain $\Omega$, $|B|$ is bounded by $|H|$. By Theorem \ref{main-theo-B}, it suffices to show that, via pinching (\ref{pinching-condition-1}), $|M|^\frac{1}{n}\|H\|_\infty$ has a uniform upper bound.
	
	Let us argue by contradiction. Let $\Omega_i$ be a sequence of convex domains in $N_i$ such that $|\partial \Omega_i|^\frac{1}{n}\|H_{\partial \Omega_i}\|_\infty \to \infty$. 
	
	By Gauss equaton, the convexity of $\Omega_i$ implies that $\partial \Omega_i$ is also a Riemannian manifold of sectional curvature $\ge \mu$. By Lemma \ref{lem-intrinsic-diam}, we have
	$$\operatorname{diam}(\partial \Omega_i)\le C(n,\mu, R) \operatorname{diam}(\Omega_i).$$
	By Bishop's volume comparison theorem, the volume $|\partial \Omega_i|$ is also uniformly bounded from above.
	
	By the discussion above, we may rescale $\Omega_i$ such that $|\partial \Omega_i|= 1$ and the rescaled sectional curvature of $\Omega_i$ still admits a uniform lower bound.  By passing to a subsequence, $\partial \Omega_i$ $GH$-converges to $Y$. According to Cheeger-Colding \cite{Cheeger-Colding2000}, the 1st-eigenvalues $\lambda_1(\partial \Omega_i)$ converge to $\lambda_1(Y)$, the 1st-eigenvalue of $Y$, which by \cite{Fukaya1987} is well-defined. Then the pinching condition
	$$n(\delta+\|H_{\partial \Omega_i}\|_\infty^2)\le \lambda_1(\partial \Omega_i)(1+\epsilon)\to \lambda_1(Y)(1+\epsilon)$$
	implies that $\|H_{\partial \Omega_i}\|_\infty$ is bounded uniformly. This contradicts to the choice of $\Omega_i$.
	
\end{proof}

{\bf Acknowledgements}. The first author was supported by China Postdoctoral Science Foundation (No.2018M641317). The second author was supported partially by National Natural Science Foundation of China [11871349], [11821101], by research funds of Beijing Municipal Education Commission and Youth Innovative Research Team of Capital Normal University.

\end{document}